\documentclass[12pt]{article}

\usepackage{amsfonts}
\usepackage [dvips]{graphics}
\usepackage[centertags]{amsmath}
\usepackage{amsfonts}
\usepackage{amssymb}
\usepackage{amsthm}
\usepackage{newlfont}
\usepackage{mathrsfs}
\usepackage [dvips]{graphics}
\usepackage[centertags]{amsmath}
\usepackage{amsfonts}
\usepackage{amssymb}
\usepackage{amsthm}
\usepackage{newlfont}
\usepackage{hyperref}
\usepackage{cases}
\newlength{\defbaselineskip}
\newcommand{\setlinespacing}[1]%
           {\setlength{\baselineskip}{#1 \defbaselineskip}}

\theoremstyle{plain}
\newtheorem{thm}{Theorem}[section]
\newtheorem{cor}[thm]{Corollary}
\newtheorem{pro}[thm]{Problem }
\newtheorem{lem}[thm]{Lemma}

\theoremstyle{definition}

\newtheorem{ass}{Assumption}[section]

\newcommand{\la}{\langle}
\newcommand{\ra}{\rangle}

\makeatletter\@addtoreset{equation}{section} \makeatother

\begin{document}
\title{Optimal Variational Principle for Backward Stochastic Control Systems Associated with L\'{e}vy
Processes \thanks{This work is partially supported by the National
Basic Research Program of China (973 Program) (Grant
No.2007CB814904), the National Natural Science Foundation of China
(Grants No.10325101, 11071069), the Specialized Research Fund for
the Doctoral Program of Higher Education of China (Grant
No.20090071120002) and the Innovation Team Foundation of the
Department of Education of Zhejiang Province (Grant No.T200924).}}

\date{}

\author{Maoning Tang$^{a}$\hspace{1cm}Qi Zhang$^{b}$\hspace{1cm}
\\ \small{$^{a}$Department of Mathematical Sciences, Huzhou University, Zhejiang 313000, China}\\
\small{Email: tmorning@hutc.zj.cn}
\\ \small{$^{b}$School of Mathematical Sciences, Fudan University, Shanghai 200433, China}
\\ \small{Email: qzh@fudan.edu.cn}}

\maketitle

\begin{abstract}

The paper is concerned with optimal control of  backward stochastic
differential equation (BSDE) driven by Teugel's martingales and an
independent multi-dimensional Brownian motion, where Teugel's
martingales are a family of pairwise strongly orthonormal
martingales associated with L\'{e}vy processes (see Nualart and
Schoutens \cite{NuSc}). We derive the necessary and sufficient
conditions for the existence of the optimal control by means of
convex variation methods and duality techniques. As an application,
the optimal control problem of linear backward stochastic
differential equation with a quadratic cost criteria (called
backward linear-quadratic problem, or BLQ problem for short) is
discussed and characterized by stochastic Hamilton system.
%In our case, the stochastic Hamilton system is a linear forward-backward
%stochastic differential equation driven by Teugel's martingales and
%an independent multi-dimensional Brownian motion, consisting of the
%state equation, adjoint equation and the dual presentation of the
%optimal control.
\end{abstract}


\maketitle

\section{ Introduction }

It is well known that the maximum principle for a stochastic optimal
control problem involves the so-called adjoint processes which solve
the corresponding adjoint equation. In fact, the adjoint equation is
in general a linear backward stochastic differential equation (BSDE)
with a specified a random terminal condition on the state. Unlike a
forward stochastic differential equation, the solution of a BSDE is
a pair of adapted solutions. Thus, in order to obtain the maximum
principle, we need first obtain the existence and uniqueness theorem
for the pair of adapted solutions of adjoint equation.

The linear BSDE was first proposed by Bismut \cite{Bism73} in 1973.
This research field developed fast after the pioneer work of Pardoux
and Peng \cite{PaPe90} in 1990 got the existence and uniqueness
theorem for the solution of nonlinear BSDE driven by Brownian motion
under Lipschitz condition. Now BSDE theory has been playing a key
role not only in dealing with stochastic optimal control problems,
but in mathematical finance, particularly in hedging and nonlinear
pricing theory for imperfect market (see e.g. \cite{KPQ97}).

As for BSDE driven by the non-continuous martingale, Tang and Li
\cite{Tali94} first discussed the existence and uniqueness theorem
of the solution of BSDE driven by Poisson point process and
consequently proved the maximum principle for optimal control of
stochastic systems with random jumps. In 2000, Nualart and
Schoutens \cite{NuSc} got a martingale representation theorem
for a type of L\'{e}vy processes through Teugel's martingales, where Teugel's martingales are a family of
pairwise strongly orthonormal martingales associated with L\'{e}vy
processes. Later, they proved in
\cite{NuSc01} the existence and uniqueness theory of BSDE driven by
Teugel's martingales. The above results are further extended to the
one-dimensional BSDE driven by Teugel's martingales and an
independent multi-dimensional Brownian motion by Bahlali et al
\cite{BEE03}. One can refer to \cite{Otm06, Otm08, ReFa09, ReOt10}
for more results on such kind of BSDEs. %driven by Teugel's
%martingales and an independent Brownian motion.

In the mean time, the stochastic optimal control problems related to
Teugel's martingales were studied. In 2008, a stochastic
linear-quadratic problem with L\'{e}vy processes was considered by
Mitsui and Tabata \cite{MiTa08}, in which they established the
closeness property of multi-dimensional backward stochastic Riccati
differential equation(BSRDE) with Teugel's martingales and proved
the existence and uniqueness of solution to such kind of
one-dimensional BSRDE, moreover, in their paper an application of
BSDE to a financial problem with full and partial observations was
demonstrated. Motivated by \cite{MiTa08}, Meng and Tang
\cite{MeTa08} studied the general stochastic optimal control problem
for the forward stochastic systems driven by Teugel's martingales
and an independent multi-dimensional Brownian motion, of which the
necessary and sufficient optimality conditions in the form of
stochastic maximum principle with the convex control domain are
obtained.

However, \cite{MeTa08} and \cite{MiTa08} are only concerned with the optimal control
problem of the forward controlled stochastic system. Since a BSDE is
a well-defined dynamic system itself and has important applications
in mathematical finance, it is necessary and natural to consider the
optimal control problem of BSDE. Actually, there has been much
literature on BSDE control system driven by Brownian motion (see
e.g. \cite{BGM10, Bah2010, DoZh99, AnZh02, AnZh01}). But to our best
knowledge, there is no discussion on the optimal control problem of
BSDE driven by Teugel martingales and an independent Brownian
motion, which motives us to write this paper.

In this paper, by means of convex variation  methods and duality
techniques, we will give the necessary and sufficient conditions for
the existence of the optimal
control for BSDE system driven by Teugel martingales and an independent multi-dimensional Brownian motion.
As an application, the optimal control for linear backward stochastic
differential equation with a quadratic cost criteria or
called backward linear-quadratic (BLQ) problem is discussed in
details. The optimal
control of BLQ problem will be characterized by stochastic Hamilton systems.
In this case, the stochastic Hamilton system is a linear forward-backward stochastic
differential equation driven by Teugel's martingales and an
independent multi-dimensional Brownian motion, consisting
of the state equation, the adjoint equation and the dual presentation of
the optimal control.

The rest of this paper is organized as follows. In section 2, we introduce useful notation and some existing results on
stochastic differential equations (SDEs) and BSDEs driven by Teugel's martingales. In
section 3, we state the optimal control problem we study, give needed assumptions and prove some
preliminary results on variational equation and variational inequality. In section 4,
we prove the necessary and sufficient optimality conditions for the optimal control problem put forward in section 3.
As an application, the optimal control for BLQ problem is discussed in section 5.

\section{Notation and preliminaries}

Let $(\Omega, \mathscr{F},\{\mathscr{F}_t\}_{0\leq t\leq T}, P)$ be
a complete probability space. The filtration
$\{\mathscr{F}_t\}_{0\leq t\leq T}$ is right-continuous and
generated by a $d$-dimensional standard Brownian motion $\{W(t),
0\leq t\leq T\}$ and a one-dimensional L\'{e}vy process $\{L(t),
0\leq t \leq T\}$. It is known that $L(t)$ has a characteristic
function of the form $$Ee^{i\theta L(t)}=\exp\bigg[ia\theta
t-{1\over2}\sigma^2\theta^2t+t\int_{\mathbb{R}^1}(e^{i\theta
x}-1-i\theta x I_{\{|x|<1\}})v(dx)\bigg],$$ where
$a\in\mathbb{R}^1$, $\sigma>0$ and $v$ is a measure on
$\mathbb{R}^1$ satisfying (i)$\displaystyle \int_0^T(1\wedge
x^2)v(dx)<\infty$ and (ii) there exists $\varepsilon >0$ and
$\lambda >0$, s.t. $\displaystyle \int_{\{-\varepsilon,
\varepsilon\}^c} e^{\lambda |x|}v(dx)<\infty$. These settings imply
that the random variables $L(t)$ have moments
of all orders. %i.e, $\displaystyle \int_\infty^\infty|x|^ivd(x)<\infty,~~~\forall i\geq2$.
Denote by $\mathscr{P}$  the
predictable sub-$\sigma$ field of $\mathscr B([0, T])\times
\mathscr{F}$, then we introduce the following notation used
throughout this paper.

$\bullet$~~$H$: a Hilbert space with norm $\|\cdot\|_H$.

$\bullet$~~$\langle\alpha,\beta\rangle:$ the inner product in
$\mathbb{R}^n, \forall \alpha,\beta\in\mathbb{R}^n.$

$\bullet$~~$|\alpha|=\sqrt{\langle\alpha,\alpha\rangle}:$ the norm
of $\mathbb{R}^n,\forall \alpha\in\mathbb{R}^n.$

$\bullet$~~$\langle A,B\rangle=tr(AB^T):$ the inner product in
$\mathbb{R}^{n\times m},\forall A,B\in \mathbb{R}^{n\times m}.$

$\bullet$~~$|A|=\sqrt{tr(AA^T)}:$ the norm of $\mathbb{R}^{n\times
m},\forall A\in \mathbb{R}^{n\times m}$.

$\bullet$~~$l^2$: the space of all real-valued sequences
$x=(x_n)_{n\geq 0}$ satisfying
$$\|x\|_{l^2}\triangleq\sqrt{\displaystyle \sum_{i=1}^\infty x_i^2}<+\infty.$$

$\bullet$~~$l^2(H):$ the space of all H-valued sequence
$f=\{f^i\}_{i\geq 1}$ satisfying
$$\|f\|_{l^2(H)}\triangleq\sqrt{\displaystyle\sum_{i=1}^\infty||f^i||_H^2}<+\infty.$$

$\bullet$~~$l_{\mathscr{F}}^2(0, T, H):$ the space of all
$l^2(H)$-valued and ${\mathscr{F}}_t$-predictable processes
$f=\{f^i(t,\omega),\ (t,\omega)\in[0,T]\times\Omega\}_{i\geq1}$
satisfying
$$\|f\|_{l_{\mathscr{F}}^2(0, T, H)}\triangleq\sqrt{E\displaystyle\int_0^T\sum_{i=1}^\infty||f^i(t)||_H^2dt}<\infty.$$

$\bullet$~~$M_{\mathscr{F}}^2(0,T;H):$ the space of all $H$-valued
and ${\mathscr{F}}_t$-adapted processes $f=\{f(t,\omega),\
(t,\omega)\in[0,T]\times\Omega\}$ satisfying
$$\|f\|_{M_{\mathscr{F}}^2(0,T;H)}\triangleq\sqrt{E\displaystyle\int_0^T\|f(t)\|_H^2dt}<\infty.$$

$\bullet$~~$S_{\mathscr{F}}^2(0,T;H):$ the space of all $H$-valued
and ${\mathscr{F}}_t$-adapted  c\`{a}dl\`{a}g processes
$f=\{f(t,\omega),\ (t,\omega)\in[0,T]\times\Omega\}$ satisfying
$$\|f\|_{S_{\mathscr{F}}^2(0,T;H)}\triangleq\sqrt{E\displaystyle\sup_{0 \leq t \leq T}\|f(t)\|_H^2dt}<+\infty.$$

$\bullet$~~$L^2(\Omega,{\mathscr{F}},P;H):$ the space of all
$H$-valued random variables $\xi$ on $(\Omega,{\mathscr{F}},P)$
satisfying
$$\|\xi\|_{L^2(\Omega,{\mathscr{F}},P;H)}\triangleq E\|\xi\|_H^2<\infty.$$

We denote by $\{H^i(t), 0\leq t \leq T\}_{i=1}^\infty$ the Teugel's
martingales
associated with the L\'{e}vy process $\{L(t),0\leq t \leq T\}$. %(see, e.g., Nualart and
%Schoutens \cite{NuSc, NuSc01}, Bahlali et al. \cite{BEE03},).
$H^i(t)$ is given by
$$
H^i(t)=c_{i,i}Y^{(i)}(t)+c_{i,i-1}Y^{(i-1)}(t)+\cdots+c_{i,1}Y^{(1)}(t),
$$
where $Y^{(i)}(t)=L^{(i)}(t)-E[L^{(i)}(t)]$ for all $i\geq 1$,
$L^{(i)}(t)$ are so called power-jump processes with
$L^{(1)}(t)=L(t)$, $L^{(i)}(t)=\displaystyle\sum_{0<s\leq t}(\Delta
L(s))^i$ for $i\geq 2$ and the coefficients $c_{ij}$ correspond to
the orthonormalization of polynomials $1,x, x^2,\cdots$ w.r.t. the
measure $\mu(dx)=x^2v(dx)+\sigma^2\delta_0(dx)$. The Teugel's
martingales $\{H^i(t)\}_{i=1}^\infty$ are pathwise strongly
orthogonal %in the sense that \underline{$H^{(i)}H^{(j)}$} is a
%martingale
and their predictable quadratic variation processes are given by
$$\langle H^{(i)}(t), H^{(j)}(t)\rangle=\delta_{ij}t£¬$$
%where $\langle H^{(i)}(t), H^{(j)}(t)\rangle$ denotes the
%predictable quadratic variational process corresponding $H^i$ and
%$H^j$.
For more details of Teugel's martingales, we invite the
reader to consult Nualart and Schoutens \cite{NuSc, NuSc01}.

In what follows, we will state some basic results on SDE and BSDE
driven by Teugel's martingales
 $\{H^i(t),0\leq t\leq T\}_{i=1}^\infty$ and  the $d$-dimensional Brownian motion
 $\{W(t), 0\leq t\leq T\}.$

Consider SDE:
\begin{equation}\label{eq:1.1}
\begin{array}{ll}
~~~~~~~~~~~~~~~~~~~~X(t)=&a+\displaystyle\int_0^tb(s,X(s))ds+\sum_{i=1}^d\int_0^tg^i(s,X(s))dW^i(s)\\
&+\displaystyle\sum_{i=1}^\infty\int_0^t\sigma^i(s,X(s-))dH^i(s),\ \
t\in [0, T],
\end{array}
\end{equation}
where  $(a,b,g,\sigma)$ are given mappings satisfying the
assumptions below.
\begin{ass}\label{ass:1.1}
Random variable $a$ is ${\mathscr{F}}_0$-measurable and
$(b,g,\sigma)$ are three random mappings

$$
b:[0,T]\times \Omega\times \mathbb{R}^n\longrightarrow \mathbb{R}^n,
$$
$$
g\equiv (g^1,g^2,\cdots,g^d):[0,T]\times \Omega\times
\mathbb{R}^n\longrightarrow \mathbb{R}^{n\times d},
$$
$$
\sigma\equiv{(\sigma^i)}_{i=1}^\infty:[0,T]\times \Omega\times
\mathbb{R}^n\longrightarrow l^2(\mathbb{R}^n)
$$
satisfying\\
(i) $ b,g $ and $\sigma$ are $ {\mathscr{P}}\bigotimes
{\mathscr B}(\mathbb{R}^n)$ measurable % w.r.t. ${\mathscr{P}}\bigotimes
%{\mathscr B}(R^n)/{\mathscr B}(R^n),{\mathscr{P}}\bigotimes {\mathscr
%B}(R^n)/{\mathscr{B}(R^{n\times d})},\\ {\mathscr{P}}\bigotimes {\mathscr
%B}(R^n)/{\mathscr{B}(l^2(R^n))}$ respectively;
with $b(\cdot,0)\in M_{\mathscr{F}}^2(0,T;\mathbb{R}^n)$,
$g(\cdot,0)\in M_{\mathscr{F}}^2(0,T;\mathbb{R}^{n\times d})$ and
$\sigma(\cdot,0)\in l_{\mathscr{F}}^2(0,T;
\mathbb{R}^n).$\\
(ii) $b,g$ and $\sigma$ are uniformly Lipschitz  continuous w.r.t.
$x$, i.e. there exists a constant  $C>0$ s.t. for all
$(t,x,\bar{x})\in [0, T]\times \mathbb{R}^n\times \mathbb{R}^n$ and
a.s. $\omega\in\Omega$,
$$
\begin{array}{ll}
|b(t,x)-b(t,\bar{x})|+|g(t,x)-g(t,\bar{x})|+||\sigma(t,x)-\sigma(t,\bar{x})||_{l^2(\mathbb{R}^n)}\leq
C|x-\bar{x}|.
\end{array}
$$
\end{ass}

\begin{lem}[{\bf\cite{TaWu09}, Existence and Uniqueness Theorem of SDE}]\label{lem:1.1}
If coefficients $(a,b,g,\sigma)$ satisfy Assumption \ref{ass:1.1},
then SDE \eqref{eq:1.1} has a unique solution $x(\cdot)\in
S_{\mathscr{F}}^2(0,T;\mathbb{R}^n)$.
\end{lem}

\begin{lem}[{\bf\cite{MeTa08}, Continuous Dependence Theorem of SDE}]  \label{lem:1.2}
Assume coefficients $(a,b,g,\sigma)$ and
$(\bar{a},\bar{b},\bar{g},\bar{\sigma})$ satisfy Assumption
\ref{ass:1.1}. If $x(\cdot)$ and $\bar{x}(\cdot)$ are the solutions
to SDE \eqref{eq:1.1} corresponding to $(a,b,g,\sigma)$ and
$(\bar{a},\bar{b},\bar{g},\bar{\sigma})$, respectively, %with the uniformly Lipschitz
%constant $C$.
%Let two data $(a,b,g,\sigma)$ and
%$(\bar{a},\bar{b},\bar{g},\bar{\sigma})$ satisfy Assumption
%\ref{ass:1.1} with the uniformly Lipschitz constant $C$.  And suppose that
%$X(\cdot)$and $\bar{X}(\cdot)$ are the solutions
%to  the SDE \eqref{eq:1.1} corresponding to  $(a,b,g,\sigma)$ and
%$(\bar{a},\bar{b},\bar{g},\bar{\sigma})$ respectively.
then we have
\begin{equation*}
\begin{array}{ll}
~E\displaystyle\sup_{0\leq t\leq T}|x(t)-\bar{x}(t)|^2\leq
&K\bigg[|a-\bar{a}|^2+E\displaystyle\int_0^T|b(t,\bar{x}(t))-\bar{b}(t,\bar{x}(t))|^2dt\\
&+E\displaystyle\int_0^T|g(t,\bar{x}(t))-\bar{g}(t,\bar{x}(t))|^2dt\\
&+E\displaystyle\int_0^T||\sigma(t,\bar{x}(t))-\bar{\sigma}
(t,\bar{x}(t))||_{l^2(\mathbb{R}^n)}^2dt\bigg],
\end{array}
\end{equation*}
where $K$ is a positive constant depending only on $T$ and the
Lipschitz constant $C$.

In particular, for
$(\bar{a},\bar{b},\bar{g},\bar{\sigma})=(0,0,0,0),$ we have
$$
\begin{array}{ll}
&E\displaystyle\sup_{0\leq t\leq T}|x(t)|^2\\
\leq&K\bigg[|a|^2+E\displaystyle \int_0^T|b(t,0)|^2dt
+E\displaystyle\int_0^T|g(t,0)|^2dt
+E\displaystyle\int_0^T||\sigma(t,0)||_{l^2(\mathbb{R}^n)}^2dt\bigg]<+\infty.
\end{array}
$$
\end{lem}
%See\cite{MeTa08} for the details of  the proof.

Now we consider BSDE:
\begin{eqnarray}\label{eq:13}
\begin{split}
y(t)
=&\xi+\displaystyle\int_t^Tf(s,y{(s)},q(s),z(s))ds-\displaystyle\sum_{i=1}^d\int_t^Tq^i(s)dW^i(s)
\\&-\displaystyle\sum_{i=1}^\infty\int_t^T z^i(s)dH^i(s),\ \
t\in [0, T],
\end{split}
\end{eqnarray}
where coefficients $(\xi,f)$ are given mappings satisfying the
assumptions below.
\begin{ass}\label{ass:1.3}
The terminal value $\xi\in L^2(\Omega,{\mathscr{F}}_T,P;
\mathbb{R}^n)$ and $f$ is a random mapping  $$f:[0,T]\times
\Omega\times \mathbb{R}^n\times \mathbb{R}^{n\times
d}\times l^2(\mathbb{R}^n)\longrightarrow \mathbb{R}^n$$ satisfying\\
(i) $f$ is ${\mathscr{P}}\bigotimes {\mathscr
B}(\mathbb{R}^n)\bigotimes {\mathscr B}(\mathbb{R}^{n\times
d})\bigotimes {\mathscr B}(l^2(\mathbb{R}^n))$ measurable with
$f(\cdot,0,0,0)\in
M_{\mathscr{F}}^2(0,T;\mathbb{R}^n)$.\\
(ii) $f$ is uniformly Lipschitz continuous w.r.t. $(y,q,z)$, i.e.
there exists a constant $C>0$ s.t. for all
$(t,y,q,z,\bar{y},\bar{q}, \bar{z})\in [0, T]\times
\mathbb{R}^n\times \mathbb{R}^{n\times d}\times
l^2(\mathbb{R}^n)\times \mathbb{R}^n\times\mathbb{R}^{n\times
d}\times l^2(\mathbb{R}^n)$ and a.s. $\omega\in\Omega$,
$$
\begin{array}{ll}
&|f(t,y,q,z)-f(t,\bar{y},\bar{q},\bar{z})|\leq
C\bigg[|y-\bar{y}|+|q-\bar{q}|+\|z-\bar{z}\|_{l^2({\mathbb{R}^n})}\bigg].
\end{array}
$$
\end{ass}

\begin{lem}[{\bf \cite{BEE03}, Existence and Uniqueness of BSDE}]\label{lem:1.3}
If coefficients $(\xi, f)$ satisfy Assumption \ref{ass:1.3}, then
BSDE \eqref{eq:13} has a unique solution $$(y(\cdot), q(\cdot),
z(\cdot))\in S_{\mathscr{F}}^2(0,T;\mathbb{R}^n)\times
M_{\mathscr{F}}^2(0,T; \mathbb{R}^{n\times d}) \times
l_{\mathscr{F}}^2(0,T;\mathbb{R}^n).$$
\end{lem}
%The similar proof can in founded in \cite{BEE03}.

\begin{lem}[{\bf \cite{BEE03}, Continuous Dependence Theorem of
BSDE}]\label{lem:1.4} Assume that coefficients $(\xi,f)$ and
$(\bar{\xi},\bar{f})$ satisfy Assumption \ref{ass:1.3}. If
$(y(\cdot),q(\cdot),z(\cdot))$ and
$(\bar{y}(\cdot),\bar{q}(\cdot),\bar{z}(\cdot))$ are the solutions
to BSDE \eqref{eq:13} corresponding to $(\xi,f)$ and
$(\bar{\xi},\bar{f})$, respectively, then we have
\begin{equation*}
\begin{array}{ll}
~&E\displaystyle\sup_{0\leq t\leq T}|y(t)-\bar{y}(t)|^2
+E\int_0^T|q(t)-\bar{q}(t)|^2dt+E\int_0^T||z(t)
-\bar{z}(t)||^2_{l^2({\mathbb{R}^n})}dt
\\\leq&K\bigg[E|\xi-\bar{\xi}|^2
+E\displaystyle\int_0^T|f(t,\bar{y}{(t)},\bar{q}(t),\bar{z}(t))
-\bar{f}(t,\bar{y}{(t)},\bar{q}(t),\bar{z}(t))|^2dt\bigg],
\end{array}
\end{equation*}
where $K$ is a positive constant depending only on $T$ and the
Lipschitz constant $C$.

In particular, if $(\bar{\xi},\bar{f})=(0,0)$, we have
\begin{equation}\label{eq:1.5}
\begin{array}{ll}
&E\displaystyle\sup_{0\leq t\leq T}|y(t)|^2
+E\int_0^T|q(t)|^2dt+E\int_0^T||z(t)||^2_{l^2({\mathbb{R}^n})}dt
\\\leq&K\bigg[E|\xi|^2
+E\displaystyle\int_0^T|f(t,0,0,0) |^2dt\bigg].
\end{array}
\end{equation}
\end{lem}
In view of Assumptions \ref{ass:1.1}-\ref{ass:1.3}, Lemmas
\ref{lem:1.1}-\ref{lem:1.4} follow from an application of It\^{o}'s
formula, Gronwall's inequality and Burkholder-Davis-Gundy
inequality. One can refer to \cite{BEE03}, \cite{MeTa08} and
\cite{TaWu09} for details.

\section{Formulation of the problem and preliminary lemmas}

Let the admissible control set $U$ be a nonempty convex subset of
$\mathbb{R}^m$. An admissible control process $u(\cdot)$ is defined
as a ${\mathscr{F}}_t$-predictable process with values in $U$ s.t.
$E\displaystyle\int_0^T|u(t)|^2dt<+\infty$. We denote by ${\cal A}$
the set including all admissible control processes.

For any given  admissible control $u(\cdot)\in{\cal A}$, we consider
the following controlled nonlinear BSDE driven by multi-dimensional
Brownian motion $W$ and Teugel's martingales $\{H^i\}_{i=1}^\infty$:
\begin{eqnarray}\label{eq:1.3}
\begin{split}
y(t)=&\xi+\displaystyle\int_t^Tf(s,y{(s)},q(s),z(s),u(s))ds\\&
-\displaystyle\sum_{i=1}^d\int_t^Tq^i(s)dW^i(s)
-\displaystyle\sum_{i=1}^\infty\int_t^T z^i(s)dH^i(s),\ \ t\in[0,T]
\end{split}
\end{eqnarray}
with the cost functional
\begin{equation}\label{eq:2.2}
J(u(\cdot))=E\displaystyle\bigg[\int_0^Tl(t,y(t),q(t),z(t),u(t))dt+\Phi(y(0))\bigg],
\end{equation}
where  $$\xi: \Omega\longrightarrow \mathbb{R}^n,$$
$$f:[0,T]\times \Omega\times
\mathbb{R}^n\times \mathbb{R}^{n\times d}\times
l^2(\mathbb{R}^n)\times U\longrightarrow \mathbb{R}^n,$$
$$l:[0,T]\times \Omega\times
\mathbb{R}^n\times \mathbb{R}^{n\times d}\times
l^2(\mathbb{R}^n)\times U\longrightarrow \mathbb{R}^1$$ and
$$\phi:\Omega\times \mathbb{R}^{n}\longrightarrow \mathbb{R}^1$$
are given coefficients.

Throughout this paper, we introduce the following basic assumptions
on coefficients $(\xi, f, l, \phi)$.
\begin{ass}\label{ass:2.1}  The terminal value $\xi\in L^2(\Omega,{\mathscr{F}}_T,P;
\mathbb{R}^n)$ and the random mapping $f$ is
${\mathscr{P}}\bigotimes {\mathscr B}(\mathbb{R}^n)\bigotimes
{\mathscr B} (\mathbb{R}^{n\times d})\bigotimes {\mathscr
B}(l^2(\mathbb{R}^n))\bigotimes {\mathscr B}(U)$ measurable with
$f(\cdot,0,0,0,0)\in M^2(0,T;\mathbb{R}^n)$. For almost all $(t,
\omega)\in [0, T]\times \Omega$, $f(t,\omega, y,p,z,u)$ is
Fr\'{e}chet differentiable w.r.t. $(y,p,z,u)$ and the corresponding
Fr\'{e}chet  derivatives $f_y, f_p, f_z, f_u$ are continuous and
uniformly bounded.
\end{ass}
\begin{ass}\label{ass:2.2}
The random mapping $l$ is ${\mathscr{P}}\bigotimes {\mathscr
B}(\mathbb{R}^n)\bigotimes {\mathscr B} (\mathbb{R}^{n\times
d})\bigotimes {\mathscr B}(l^2(\mathbb{R}^n))\bigotimes {\mathscr
B}(U)$ measurable and for almost all $(t, \omega)\in [0, T]\times
\Omega$, $l$ is Fr\'{e}chet differentiable w.r.t. $(y,p,z,u)$ with
continuous Fr\'{e}chet derivatives $l_y, l_q, l_z, l_u$. The random
mapping $\phi$ is ${\mathscr{F}}_T \bigotimes {\mathscr
B}(\mathbb{R}^n)$ measurable and for almost all $(t,\omega)\in [0,
T]\times\Omega$, $\phi$ is Fr\'{e}chet differentiable w.r.t. $y$
with continuous Fr\'{e}chet derivative $\phi_y$. Moreover, for
almost all $(t,\omega)\in [0, T]\times \Omega$, there exists a
constant $C$ s.t. for all $(p,q,z,u)\in
\mathbb{R}^n\times\mathbb{R}^{n\times d}\times
l^2(\mathbb{R}^n)\times U$,
$$|l|\leq C(1+|y|^2+|q|^2+|z|^2+|u|^2),\ \ |\phi|\leq C(1+|y|^2),$$
$$|l_y|+|l_q|+|l_z|+|l_u|\leq C(1+|y|+|q|+|z|+|u|)\ and\ |\phi_y|\leq
C(1+|y|).$$
\end{ass}
Under Assumption \ref{ass:2.1}, we can get from Lemma \ref{lem:1.3}
that for each  $u(\cdot)\in {\cal A}$, the system \eqref{eq:1.3}
admits a unique strong solution. We denote the strong solution of
\eqref{eq:1.3} by $(y^u(\cdot), q^u(\cdot), z^u(\cdot))$, or
$(y(\cdot), q(\cdot), z(\cdot))$ if its dependence on admissible
control $u(\cdot)$ is clear from context. Then we call $(y(\cdot),
q(\cdot), z(\cdot))$ the state processes corresponding to the
control process $u(\cdot)$ and call $(u(\cdot); y(\cdot), q(\cdot),
z(\cdot))$ the admissible pair. Furthermore, by Assumption
\ref{ass:2.2} and a priori estimate \eqref{eq:1.5}, it is easy to
check that
$$ |J(u(\cdot))|<\infty.$$

Then we put forward the optimal control problem we study.
\begin{pro}\label{pro:2.1}
Find an admissible control $\bar{u}(\cdot)$ such that
\begin{equation*}\label{eq:b7}
J(\bar{u}(\cdot))=\displaystyle\inf_{u(\cdot)\in {\cal
A}}J(u(\cdot)).
\end{equation*}
\end{pro}
Any  $\bar{u}(\cdot)\in {\cal A}$ satisfying above is called an
optimal control process of Problem \ref{pro:2.1} and the
corresponding state processes $(\bar{y}(\cdot), \bar {q}(\cdot),
\bar{z}(\cdot))$ are called the optimal state processes.
Correspondingly $(\bar{u}(\cdot); \bar{y}(\cdot), \bar {q}(\cdot),
\bar{z}(\cdot))$ is called an optimal pair of Problem
\ref{pro:2.1}.\\

%\section{Preliminary lemmas }
Before we deduce the necessary and sufficient conditions for the
optimal control of Problem \ref{pro:2.1}, we need do some
preparations. Since the control domain $U$ is convex, the classical
method to get necessary conditions for optimal control processes is
the so-called convex perturbation method. More precisely, assuming
that $(\bar{u}(\cdot);\bar{y}(\cdot), \bar {q}(\cdot),
\bar{z}(\cdot))$ is an optimal pair of Problem \ref{pro:2.1}, for
any given admissible control ${u}(\cdot)$, we define an admissible
control in the form of convex variation
\begin{equation*}
u^\varepsilon(\cdot)=\bar{u}(\cdot)+\varepsilon(u(\cdot)-\bar{u}(\cdot)),
\end{equation*}
where $\varepsilon>0$ can be chosen sufficiently small. Denoting by
$(y^\varepsilon(\cdot), q^\varepsilon(\cdot), z^\varepsilon(\cdot))$
the state processes of the control system \eqref{eq:1.3}
corresponding to the control process $u^\varepsilon(\cdot)$, we
obtain the variational inequality
\begin{equation*}
J(u^\varepsilon(\cdot))- J(\bar{u}(\cdot))\geq 0.
\end{equation*}
%Then by some suitable duality relation, we can get a stochastic
%maximum principle of Pontryagin type. To this end, we need the
%following preliminary lemmas.

In what follows, we do some estimates on the optimal pair and the
convex variable pair.
\begin{lem} \label{lem:3.2}
Under Assumptions \ref{ass:2.1}-\ref{ass:2.2}, we have
\begin{eqnarray*}
\begin{split}
E\sup_{0\leq t\leq
T}|y^\varepsilon(t)-\bar{y}(t)|^2+E\int_0^T|q^\varepsilon(t)-\bar{q}(t)|^2dt+E\int_0^T||z^\varepsilon(t)
-\bar{z}(t)||^2_{l^2({\mathbb{R}^n})}dt=O(\varepsilon^2).
\end{split}
\end{eqnarray*}
\end{lem}
\begin{proof} By continuous dependence theorem of
BSDE (Lemma \ref{lem:1.4}) and the uniformly bounded property of
Fr\'{e}chet derivative $f_u$, we have
\begin{eqnarray*}
\begin{split}
&E\sup_{0\leq t\leq
T}|y^\varepsilon(t)-\bar{y}(t)|^2+E\int_0^T|q^\varepsilon(t)-\bar{q}(t)|^2dt+E\int_0^T||z^\varepsilon(t)
-\bar{z}(t)||^2_{l^2({\mathbb{R}^n})}dt
\\\leq&K E\displaystyle\int_0^T|f(t, \bar{y}(t), \bar{q}(t),
\bar{z}(t),  {u}^\varepsilon(t)) -f(t, \bar{y}(t), \bar{q}(t),
\bar{z}(t),  \bar{u}(t))\big|^2dt
\\\leq&
KE\displaystyle\int_0^T|u^\varepsilon(t)-\bar{u}(t)|^2dt
\\=&KE\displaystyle \int_0^T|(\bar{u}(t)+\varepsilon(u(t)-\bar{u}(t))-\bar{u}(t))|^2dt\\
=&K\varepsilon^2E\displaystyle\int_0^T|u(t)-\bar{u}(t)|^2dt=O(\varepsilon^2).
\end{split}
\end{eqnarray*}
Here and in the rest of this paper, $K$ is a generic positive
constant and might change from line to line.
\end{proof}
\vspace{1mm}

Then we consider the following linear BSDE served as a variational
equation:
\begin{numcases}{}\label{eq:32}
dY_t=-\bigg[f_y(t,\bar{y}(t),\bar{q}(t),\bar{z}(t),\bar{u}_t)
Y(t)+f_q(t,\bar{y}(t),\bar{q}(t),\bar{z}(t),\bar{u}_t)Q_t\nonumber\\
~~~~~~~~~~~~+f_z(t,\bar{y}(t),\bar{q}(t),\bar{z}(t),\bar{u}_t)Z(t)
+f_u(t,\bar{y}(t),\bar{q}(t),\bar{z}(t),\bar{u}_t)(u(t)-\bar{u}(t))\bigg]dt
\nonumber\\~~~~~~~~\displaystyle\sum_{i=1}^d\int_t^TQ^i(s)dW^i(s)+
\displaystyle\sum_{i=1}^\infty Z^i(t)dH^i(t)\\Y(T)=0.\nonumber
\end{numcases}
Under Assumption \ref{ass:2.1}, by Lemma \ref{lem:1.3} we know that
BSDE (\ref{eq:32}) has a unique solution  $$(Y, Q, Z)\in
S_{\mathscr{F}}^2(0,T;\mathbb{R}^n)\times M_{\mathscr{F}}^2(0,T;
\mathbb{R}^{n\times d}) \times
l_{\mathscr{F}}^2(0,T;\mathbb{R}^n).$$
%Moreover, we have
\begin{lem}\label{eq:3.2}
Under Assumptions \ref{ass:2.1}-\ref{ass:2.2}, it follows that
\begin{eqnarray*}
\begin{split}
&E\displaystyle \sup_{0\leq t\leq
T}|y^\varepsilon(t)-\bar{y}(t)-\varepsilon
Y(t)|^2+E\int_0^T|q^\varepsilon(t)-\bar{q}(t)-\varepsilon Q(t)|^2dt
\\&+E\int_0^T||z^\varepsilon(t) -\bar{z}(t)-\varepsilon
Z(t)||^2_{l^2({\mathbb{R}^n})}dt=o(\varepsilon^2).
\end{split}
\end{eqnarray*}
\end{lem}
\begin{proof}
Firstly, one can check that
$$
\begin{array}{ll}
&y^\varepsilon(t)-\bar{y}(t)\\
=&\displaystyle\int_t^T\bigg[{f}_y^\varepsilon(s)
(y^\varepsilon(s)-\bar{y}(s))+{f}_q^\varepsilon(s)
(q^\varepsilon(s)-\bar{q}(s))\\&\ \ \ \ \ \ \ \
+{f}_z^\varepsilon(s)
(z^\varepsilon(s)-\bar{z}(s))+{f}_u^\varepsilon(s)
(u^\varepsilon(s)-\bar{u}(s))\bigg]ds
\\&\ \ \ \ \ -\displaystyle\sum_{i=1}^d\int_t^T\big(q^{i\varepsilon}(s)-
\bar{q}^i(s)\big)dW^i(s)-\displaystyle\sum_{i=1}^\infty\int_t^T
\big(z^{i\varepsilon}(s)-\bar{z}^i(s)\big)dH^i(s)
\end{array}
$$
and
\begin{eqnarray*}
  \begin{split}
  \varepsilon Y(t)=& \displaystyle \displaystyle\int_t^T\bigg[{f}_y(s)
\varepsilon{Y}(s)++{f}_q(s)\varepsilon  Q(s)+{f}_z(s)\varepsilon
Z(s)+{f}_u(s) \varepsilon(u(s)-\bar{u}(s))\bigg]ds
\\&\ \ \ \ \ -\displaystyle\sum_{i=1}^d\int_t^T\varepsilon
Q^i(s)dW^i(s)-\displaystyle\sum_{i=1}^\infty\int_t^T \varepsilon
Z^i(s)dH^i(s),
  \end{split}
\end{eqnarray*}
where we have used the abbreviations for $\varphi = f,l$ as follows:
\begin{numcases}{}\label{tz1}
\varphi_y(t)=\varphi_y(t,\bar{y}(t),\bar{q}(t),\bar{z}(t),\bar{u}_t),
\nonumber\\\varphi_z(t)=\varphi_z(t,\bar{y}(t),\bar{q}(t),\bar{z}(t),\bar{u}_t),
\nonumber\\\varphi_q(t)=\varphi_q(t,\bar{y}(t),\bar{q}(t),\bar{z}(t),\bar{u}_t),
\nonumber\\\varphi_u(t)=\varphi_u(t,\bar{y}(t),\bar{q}(t),\bar{z}(t),\bar{u}_t),
\\\tilde{\varphi}_y^\varepsilon(t)=
\displaystyle\int_0^1\varphi_y(t,\bar{y}(t)
+\lambda(y^\varepsilon(t)-\bar{y}(t)),\bar{z}(t)+\lambda(z^\varepsilon(t)-\bar{z}(t)),\bar{u}(t)
+\lambda(u^\varepsilon(t)-u(t)))d\lambda,
\nonumber\\\tilde{\varphi}_z^\varepsilon(t)=
\displaystyle\int_0^1\varphi_z(t,\bar{y}(t)
+\lambda(y^\varepsilon(t)-\bar{y}(t)),\bar{z}(t)+
\lambda(z^\varepsilon(t)-\bar{z}(t)),\bar{u}(t)
+\lambda(u^\varepsilon(t)-u(t)))d\lambda,
\nonumber\\\tilde{\varphi}_q^\varepsilon(t)=
\displaystyle\int_0^1\varphi_q(t,\bar{y}(t)
+\lambda(q^\varepsilon(t)-\bar{y}(t)),\bar{z}(t)+
\lambda(q^\varepsilon(t)-\bar{z}(t)),\bar{u}(t)
+\lambda(u^\varepsilon(t)-u(t)))d\lambda,
\nonumber\\\tilde{\varphi}_u^\varepsilon(t)=
\displaystyle\int_0^1\varphi_u(t,\bar{y}(t)
+\lambda(y^\varepsilon(t)-\bar{y}(t)),\bar{z}(t)+
\lambda(z^\varepsilon(t)-\bar{z}(t)),\bar{u}(t)
+\lambda(u^\varepsilon(t)-u(t)))d\lambda.\nonumber
\end{numcases}
Thus by Lemma \ref{lem:1.4} again, we get
\begin{eqnarray}\label{tz5}
\begin{array}{ll}
&E\displaystyle \sup_{0\leq t\leq
T}|y^\varepsilon(t)-\bar{y}(t)-\varepsilon
Y(t)|^2+E\int_0^T|q^\varepsilon(t)-\bar{q}(t)-\varepsilon Q(t)|^2dt
\\&~+E\displaystyle\int_0^T||z^\varepsilon(t) -\bar{z}(t)-\varepsilon
Z(t)||^2_{l^2({\mathbb{R}^n})}dt\\\leq&K\varepsilon^2\bigg[E\displaystyle\int_0^T
\bigg
|(\tilde{f}^\varepsilon_y(t)-f_y(t))Y(t)+(\tilde{f}^\varepsilon_q(t)-f_q(t))Q(t)+(\tilde{f}^\varepsilon_z(t)-f_z(t))Z(t)
\\&~~~~~~~~~~~~~~~~~+
(\tilde{f}^\varepsilon_u(t)-f_u(t))(u(t)-\bar {u}(t)) \bigg|^2dt
 \bigg]
\\=&K\varepsilon^2\cdot\alpha(\varepsilon),
\end{array}
\end{eqnarray}
where
\begin{eqnarray*}
  \begin{split}
\alpha(\varepsilon)=E\displaystyle&\int_0^T \bigg
|(\tilde{f}^\varepsilon_y(t)-f_y(t))Y(t)+(\tilde{f}^\varepsilon_q(t)-f_q(t))Q(t)
\\&\ \ \ \ \ \ \ +(\tilde{f}^\varepsilon_z(t)-f_z(t))Z(t)+
(\tilde{f}^\varepsilon_u(t)-f_u(t))(u(t)-\bar {u}(t)) \bigg|^2dt.
  \end{split}
\end{eqnarray*}
Consequently, using Lemma \ref{lem:3.2} and Assumption
\ref{ass:2.1}, by the dominated convergence theorem we can deduce
$$
\displaystyle\lim_{\varepsilon\rightarrow 0}\alpha(\varepsilon)=0.
$$
Then the lemma follows from above and (\ref{tz5}).
\end{proof}
\vspace{1mm}

\begin{lem} \label{lem:3.3}
Under Assumptions \ref{ass:2.1}-\ref{ass:2.2}, using the
abbreviations (\ref{tz1}) we have
\begin{equation*}
\begin{array}{ll}
J(u^\varepsilon(\cdot))-J(\bar{u}(\cdot))=&\varepsilon
E\phi_y(\bar{y}(0)) Y(0)+\varepsilon E\displaystyle\int_0^Tl_y(t)
Y(t)dt+\varepsilon E\displaystyle\int_0^Tl_q(t) Q(t)dt
\\&+\varepsilon E\displaystyle\int_0^Tl_z(t) Z(t)dt
+\varepsilon E\displaystyle\int_0^Tl_u(t)
(u(t)-\bar{u}(t))dt+o(\varepsilon).
\end{array}
\end{equation*}
\end{lem}
\begin{proof}
After a first order  development, we have
\begin{align*}
&J(u^\varepsilon(\cdot))-J(\bar{u}(\cdot))
\\%=&E[\phi(y^\varepsilon(0))-\phi(\bar{y}(0))]
%\\+E\displaystyle\int_0^T[l(t,y^\varepsilon(t),q^\varepsilon(t),
%z^\varepsilon(t),
%{u^\varepsilon}(t))-l(t,\bar{y}(t),\bar{q}(t),\bar{z}(t),\bar{u}_t))]dt
%\\
=&E\displaystyle\int_0^1\phi_y(\bar{y}(0)
+\lambda(y^\varepsilon(0)-\bar{y}(0)))(y^\varepsilon(0)-\bar{y}(0))d\lambda
\\&+E\displaystyle\int_0^T \tilde{l}_y^\varepsilon
(t)(y^\varepsilon(t)-\bar{y}(t))dt+E\displaystyle\int_0^T
\tilde{l}_q^\varepsilon
(t)(q^\varepsilon(t)-\bar{q}(t))dt\\&+E\displaystyle\int_0^T
\tilde{l}_z^\varepsilon
(t)(z^\varepsilon(t)-\bar{z}(t))dt+E\displaystyle\int_0^T
\tilde{l}_u^\varepsilon (t)(u^\varepsilon(t)-\bar{u}(t))dt
\\=&\varepsilon E\phi_y(\bar{y}(0))
Y(0)+E\phi_y(\bar{y}(0))(y^\varepsilon(0)-\bar{y}(0)-\varepsilon
Y(0))
\\&+E\displaystyle\int_0^1\bigg[\phi_y(\bar{y}(0)
+\lambda(y^\varepsilon(0)-\bar{y}(0)))-
\phi_y(\bar{y}(0))\bigg](y^\varepsilon(0)-\bar{y}(0))d\lambda
\\&+\varepsilon E\displaystyle\int_0^T l_y (t)
Y(t)dt+E\displaystyle\int_0^T l_y (t)
(y^\varepsilon(t)-\bar{y}(t)-\varepsilon Y(t))dt
\\&+E\displaystyle\int_0^T (\tilde{l}_y^\varepsilon
(t)-l_y (t))(y^\varepsilon(t)-\bar{y}(t))dt
\\&+\varepsilon E\displaystyle\int_0^T l_q (t)
q(t)dt+E\displaystyle\int_0^T l_q (t)
(q^\varepsilon(t)-\bar{q}(t)-\varepsilon Q(t))dt
\\&+E\displaystyle\int_0^T (\tilde{l}_q^\varepsilon
(t)-l_q (t))(q^\varepsilon(t)-\bar{q}(t))dt
\\&+\varepsilon E\displaystyle\int_0^T l_z (t)
Z(t)dt+E\displaystyle\int_0^T l_z (t)
(z^\varepsilon(t)-\bar{z}(t)-\varepsilon Z(t))dt
\\&+E\displaystyle\int_0^T (\tilde{l}_z^\varepsilon
(t)-l_z (t))(z^\varepsilon(t)-\bar{z}(t))dt
\\&+E\displaystyle\int_0^T l_u (t)\varepsilon (u(t)-\bar{u}(t))dt
+E\displaystyle\int_0^T (\tilde{l}_u^\varepsilon (t)-l_u
(t))\varepsilon (u(t)-\bar{u}(t))dt
\\=&\varepsilon E\phi_y(\bar{y}(0))
Y(0) +\varepsilon E\displaystyle\int_0^1l_y(t)Y(t)dt+\varepsilon
E\displaystyle\int_0^1l_q(t)Q(t)dt
\\&+\varepsilon
E\displaystyle\int_0^1l_z(t)Z(t)dt +\varepsilon
E\displaystyle\int_0^1l_u(t) (u(t)-\bar{u}(t))dt+\beta(\varepsilon),
\end{align*}
where $\beta(\varepsilon)$ is given by
$$
\begin{array}{ll}
\beta(\varepsilon)=&E\phi_y(\bar{y}(0))(y^\varepsilon(0)-\bar{y}(0)-\varepsilon
Y(0))
\\&+E\displaystyle\int_0^1\bigg[\phi_y(\bar{y}(0)
+\lambda(y^\varepsilon(0)-\bar{y}(0)))-
\phi_y(\bar{y}(0))\bigg](y^\varepsilon(0)-\bar{y}(0))d\lambda
\\&+E\displaystyle\int_0^T l_y (t)
(y^\varepsilon(t)-\bar{y}(t)-\varepsilon
Y(t))dt+E\displaystyle\int_0^T (\tilde{l}_y^\varepsilon (t)-l_y
(t))(y^\varepsilon(t)-\bar{y}(t))dt
\\&+E\displaystyle\int_0^T l_q (t)
(q^\varepsilon(t)-\bar{q}(t)-\varepsilon
Q(t))dt+E\displaystyle\int_0^T (\tilde{l}_q^\varepsilon (t)-l_q
(t))(q^\varepsilon(t)-\bar{q}(t))dt
\\&+E\displaystyle\int_0^T l_z (t)
(z^\varepsilon(t)-\bar{z}(t)-\varepsilon
Z(t))dt+E\displaystyle\int_0^T (\tilde{l}_z^\varepsilon (t)-l_z
(t))(z^\varepsilon(t)-\bar{z}(t))dt
\\&+E\displaystyle\int_0^T (\tilde{l}_u^\varepsilon (t)-l_u
(t))\varepsilon (u(t)-\bar{u}(t))dt.
\end{array}
$$
Thus combining Lemma \ref{lem:3.2}, Lemma \ref{lem:3.3} and
Assumption \ref{ass:2.2}, by the dominated convergence theorem we
conclude that $\beta(\varepsilon)=o(\varepsilon)$.
\end{proof}
\vspace{1mm}

By Lemma \ref{lem:3.3} and the fact that  $\displaystyle
\lim_{\varepsilon\rightarrow
0^+}\frac{J(u^\varepsilon)-J(\bar{u})}{\varepsilon}\geq 0$, we can
further deduce
\begin{cor} Under Assumptions \ref{ass:2.1}-\ref{ass:2.2}, we have the variation
inequality below
 \begin{equation}\label{eq:4.4}
\begin{array}{ll}
&E\phi_y(\bar{y}(0)) Y(0)+ E\displaystyle\int_0^Tl_y(t) Y(t)dt
+E\displaystyle\int_0^Tl_y(t) Y(t)dt\\&+
E\displaystyle\int_0^Tl_z(t) Z(t)dt+ E\displaystyle\int_0^Tl_u(t)
(u(t)-\bar{u}(t))dt\geq 0.
\end{array}
\end{equation}
\end{cor}
%
%\begin{proof} The variation inequality \eqref{eq:4.4} follows immediately from .
%\end{proof}
%\vspace{1mm}

\section {Necessary and sufficient optimality conditions}

We first introduce the adjoint equation corresponding to the
variational equation \eqref{eq:32}:
\begin{numcases}{}\label{eq:41}
dk(t)=-\bigg[-f_y^*(t)k(t)+l_y(t)\bigg]dt- \displaystyle\sum_{i=1}^d
\bigg[-f_{q^i}^*(t)k(t)+l_{q^i}(t)\bigg]dW^i(t)\nonumber\\\
\ \ \ \ \ \ \ \ \ \
-\displaystyle\sum_{i=1}^\infty\bigg[-f_{z^i}^*(t)k(t)
+l_{z^i}(t)\bigg]dH^i(t)\\
k(0)=-\phi_y(\bar{y}(0)), ~~~~0\leq t\leq T,\nonumber
\end{numcases}
where $f_y^*, f_{q^i}^* $and $f_{z^i}^*$  are the dual operators of
$f_y, f_{q^i}$ and $f_{z^i}$, respectively.

Under Assumptions \ref{ass:2.1}-\ref{ass:2.2}, by Lemma
\ref{lem:1.1} it is easy to see that the above adjoint equation has
a unique solution $k(\cdot)\in {\cal S}^2_{\mathscr{F}}(0,T;
\mathbb{R}^n)$. Then we define the Hamiltonian function
$H:[0,T]\times \mathbb{R}^n\times
 \mathbb{R}^{n\times d}\times l^2(\mathbb{R}^n)\times U\times \mathbb{R}^n \rightarrow \mathbb{R}^1$ by
\begin{equation}\label{eq:4.2}
\begin{array}{ll} \displaystyle
H(t,y,q,z,u,k)=\langle k, -f(t,y,q,z,u)\rangle+l(t,y,q,z,u)
\end{array}
\end{equation}
and rewrite the adjoint equation in the Hamiltonian system form:
\begin{equation}\label{eq:4.5}
\left\{\begin{array}{ll}
dk(t)=-H_y(t,\bar{y}(t),\bar{q}(t),\bar{z}(t),\bar{u}(t),k(t))dt\\
\ \ \ \ \ \ \ \ \ \ \ -\displaystyle\sum_{i=1}^d H_q^{i}
(t,\bar{y}(t),\bar{q}(t),\bar{z}(t),\bar{u}(t),k(t))dW^i(t)\\
\ \ \ \ \ \ \ \ \ \ \ -\displaystyle\sum_{i=1}^\infty H_{z^i}(t,\bar{y}(t),\bar{q}(t),\bar{z}(t),\bar{u}(t),k(t))dH^i(t)\\
k(0)=-\phi_y(\bar{y}(0)).
\end{array}\right.
\end{equation}

Now we are ready to give the necessary conditions for an optimal
control of Problem \ref{pro:2.1}.
\begin{thm} Under Assumptions \ref{ass:2.1}-\ref{ass:2.2},
if $(\bar{u}(\cdot); \bar{y}(\cdot),\bar{q}(\cdot), \bar{z}(\cdot))$
is  an optimal pair of  Problem \ref{pro:2.1}, then we have
\begin{equation}\label{eq:44}
H_u(t,\bar{y}(t-),\bar{q}(t),\bar{z}(t),\bar{u}(t),k(t-))
(u-\bar{u}(t))\geq 0,~\forall u\in U,\ \ a.e.\ a.s.,
\end{equation}
where $k(\cdot)$ is the solution to the adjoint equation
\eqref{eq:41}.
\end{thm}
\begin{proof} By \eqref{eq:32} and \eqref{eq:41}, applying It\^{o} formula
to $\langle Y(t), k(t) \rangle$ we have
$$
\begin{array}{ll}
&E\phi_y(\bar{y}(0)) Y(0)+ E\displaystyle\int_0^Tl_y(t) Y(t)dt
 + E\displaystyle\int_0^Tl_z(t) Z(t)dt+E\displaystyle\int_0^Tl_u(t)
(u(t)-\bar{u}(t))dt\\
=&-E\displaystyle\int_0^T\langle
k(t),f_u(t,\bar{y}(t),\bar{q}(t),\bar{z}(t),
\bar{u}_t)(u(t)-\bar{u}(t))\rangle dt+E\displaystyle\int_0^Tl_u(t)
  (u(t)-\bar{u}(t))dt.
\end{array}
$$
Then noticing the definition of Hamilton function \eqref{eq:4.2} and
the variational inequality \eqref{eq:4.4}, for any $u(\cdot)\in{\cal
A}$, we have
$$
E\displaystyle\int_0^TH_u(t,\bar{y}(t),\bar{q}(t),
\bar{z}(t),\bar{u}(t),k(t)) (u(t)-\bar{u}(t))dt\geq 0,
$$
%So for any $u\in U$, we have
%$$
%H_u(t,t,\bar{y}(t-),\bar{q}(t),\bar{z}(t),\bar{u}(t),k(t-))
%(u-\bar{u}(t))\geq 0\ \ a.e.\ a.s.
%$$
which implies (\ref{eq:44}).
\end{proof}
\vspace{1mm}

%\section {Sufficient optimality conditions}

We then consider the sufficient conditions for an optimal control of
Problem \ref{pro:2.1}.
\begin{thm}
Under Assumptions \ref{ass:2.1}-\ref{ass:2.2}, let $(\bar{u}(\cdot);
\bar{y}(\cdot), \bar{q}(\cdot), \bar{z}(\cdot))$ be an admissible
pair and ${k}(\cdot)$ be the unique solution of the corresponding
adjoint equation \eqref{eq:4.5}. Assume that for almost all
$(t,\omega)\in [0,T]\times \Omega$ , $H(t,y, q, z, u, {k}(t))$ and
$\phi(y)$ are convex w.r.t. $(y,q,z,u)$ and $y$, respectively, and
the optimality condition
\begin{equation*}
H(t,\bar{y}(t),\bar{q}(t),\bar{z}(t),\bar{u}(t),k(t))
 =\displaystyle \min_{u\in U}H(t,\bar{y}(t),\bar{q}(t),\bar{z}(t),
 u,k(t))
\end{equation*}
holds, then $(\bar{u}(\cdot);  \bar y(\cdot),\bar q(\cdot), \bar
z(\cdot) )$ is an optimal
 pair of Problem \ref{pro:2.1}.
\end{thm}
\begin{proof} Let $(u(\cdot);   y(\cdot),q(\cdot), z(\cdot))$
 be an arbitrary admissible pair. It follows from
  the form of the cost functional \eqref{eq:2.2} that
\begin{eqnarray}\label{eq:5.10}
&&J(u(\cdot))-J(\bar{u}(\cdot))\nonumber\\
&=&E\displaystyle\int_0^T\bigg[l(t,y(t),q(t),z(t),u(t))-l(t,\bar{y}(t),\bar{q}(t),
\bar{z}(t),\nonumber
\bar{u}(t))\bigg]dt+E\bigg[\phi(y(0))-\phi(\bar{y}(0))\displaystyle\bigg]
\\&=&I_1+I_2,
\end{eqnarray}
where
\begin{equation*}\label{eq:5.3} \displaystyle
I_1=E\int_0^T\bigg[l(t,y(t),q(t),z(t),u(t))-l(t,\bar y(t),\bar
q(t),\bar z(t), \bar u(t))\bigg]dt
\end{equation*}
and
\begin{equation*}
I_2=E\bigg[\phi(y(0))-\phi(\bar{y}(0))\bigg].
\end{equation*}
Due to the convexity of $\phi$, applying It\^{o} formula to $\langle
{k}(t),y(t)-\bar{y}(t)\rangle$,  we have
\begin{eqnarray}\label{eq:5.5}
\begin{split} I_2=&E[\phi(y(0))-\phi(\bar{y}(0))]
\geq E[\langle\phi_y(\bar{y}(0)),y(0)-\bar{y}(0)\rangle]=-E[\langle{k}(0),y(0)-\bar{y}(0)\rangle]\\
=&-E\displaystyle\int_0^T\langle
H_y(t,\bar{y}(t),\bar{q}(t),\bar{z}(t),\bar{u}(t), {k}(t)),
y(t)-\bar{y}(t) \rangle dt
\\&-\sum_{i=1}^dE\displaystyle\int_0^T\langle H_q^i(t,\bar{y}(t),\bar{q}(t),\bar{z}(t),\bar{u}(t), {k}(t)), q^i(t)-\bar q^i(t)\rangle dt
\\&-\displaystyle\sum_{i=1}^\infty E\displaystyle\int_0^T \langle H_z^i(t,\bar{y}(t),\bar{q}(t),\bar{z}(t),\bar{u}(t), {k}(t)),
z^i(t)-\bar{z}^i(t)\rangle dt
\\&-E\displaystyle\int_0^T\langle
f(t,{y}(t),{q}(t),{z}(t),{u}(t)) -
f(t,\bar{y}(t),\bar{q}(t),\bar{z}(t),\bar{u}(t)), k(t)\rangle dt
\\=&-J_1+J_2,
\end{split}
\end{eqnarray}
where
$$
\begin{array}{ll}
J_1=&E\displaystyle\int_0^T\langle
H_y(t,\bar{y}(t),\bar{q}(t),\bar{z}(t), \bar{u}(t), {k}(t)),
y(t)-\bar{y}(t) \rangle dt
\\&+\displaystyle\sum_{i=1}^dE\displaystyle\int_0^T\langle H_q^i(t,\bar{y}(t),\bar{q}(t),\bar{z}(t),\bar{u}(t), {k}(t)),q^i(t)-\bar q^i(t)\rangle dt
\\&+\displaystyle\sum_{i=1}^\infty E\displaystyle\int_0^T \langle
H_z^i(t,\bar{y}(t),\bar{q}(t),\bar{z}(t),\bar{u}(t), {k}(t)),
z^i(t)-\bar{z}^i(t)\rangle dt
\end{array}
$$
and
$$J_2=-E\displaystyle\int_0^T\langle f(t,{y}(t),{q}(t),{z}(t),{u}(t))
-f(t,\bar{y}(t),\bar{q}(t),\bar{z}(t),\bar{u}(t)), k(t)\rangle dt.$$
Using the definition of the Hamiltonian function \eqref{eq:4.2}
again, we have
\begin{equation}\label{eq:5.6}
\begin{array}{ll}
I_1&=E\displaystyle\int_0^T\bigg[l(t,y(t),q(t),z(t),u(t))-l(t,\bar{y}(t),\bar{q}(t), \bar z(t), \bar u(t))\bigg]dt\\
&=E\displaystyle\int_0^T\bigg[H(t,y(t),q(t),z(t),u(t),k(t))
-H(t,\bar y(t), \bar q(t), \bar z(t),\bar u(t), k(t))\bigg]dt
\\&\ \ \ +E\displaystyle\int_0^T\langle f(t,{y}(t),{q}(t),{z}(t),{u}(t))
-f(t,\bar{y}(t),\bar{q}(t),\bar{z}(t),\bar{u}(t)), k(t)\rangle dt
\\&=J_3-J_2,
\end{array}
\end{equation}
where
\begin{equation}\label{eq:5.7}
\begin{array}{ll}
J_3=E\displaystyle\int_0^T\bigg[H(t,y(t),q(t),z(t),u(t),k(t))
-H(t,\bar y(t), \bar q(t), \bar z(t),\bar u(t), k(t))\bigg]dt.
\end{array}
\end{equation}
Since  $H(t,y,q,z,u,{k}(t))$ is
convex w.r.t. $(y,q,z,u)$ for almost all $(t,\omega)\in [0, T]\times
\Omega$, it turns out that
\begin{eqnarray}\label{tz3}
  \begin{split}
    &H(t,y(t),q(t),z(t),u(t),{k}(t))
-H(t,\bar{y}(t),\bar{q}(t),\bar{z}(t),\bar{u}(t),{k}(t))
\\\geq&\langle H_y(t,\bar{y}(t),\bar{q}(t),\bar{z}(t),\bar{u}(t),{k}(t)),
y(t)-\bar{y}(t)\rangle
\\&+\sum_{i=1}^d\langle H_q^i(t,\bar{y}(t),\bar{q}(t),\bar{z}(t),\bar{u}(t),{k}(t)),q^i(t)-\bar{q}^i(t)\rangle \\
&+\sum_{i=1}^\infty\langle
H_z^i(t,\bar{y}(t),\bar{q}(t),\bar{z}(t),\bar{u}(t),{k}(t)),
z^i(t)-\bar{z}^i(t)\rangle \\&+\langle
H_u(t,\bar{y}(t),\bar{q}(t),\bar{z}(t),\bar{u}(t),{k}(t)),
u(t)-\bar{u}(t)\rangle,\ \ a.s.\ a.e.
  \end{split}
\end{eqnarray}

On the other hand, for almost all $(t, \omega)\in [0, T]\times
\Omega$, $u\rightarrow
  H(t,\bar{y}(t),\bar{q}(t),\bar{z}(t),u,k(t))$ takes its
minimal value at $\bar{u}(t)$ in the domain $U$, thus
\begin{equation} \label{eq:5.9}
\begin{array}{ll}
\langle
H_u(t,\bar{y}(t),\bar{q}(t),\bar{z}(t),\bar{u}(t),k(t)),u(t)-\bar{u}(t)\rangle
\geq 0,\ \ a.s.\ a.e.
\end{array}
\end{equation}
Therefore, by \eqref{eq:5.7}--\eqref{eq:5.9} we first have
\begin{equation} \label{eq:5.11}
\begin{array}{ll}
J_3\geq J_1.
\end{array}
\end{equation}
By \eqref{eq:5.11}, together with \eqref{eq:5.10}--\eqref{eq:5.6},
it follows that
$$
\begin{array}{ll}
J(u(\cdot))-J(\bar{u}(\cdot))=I_1+I_2=(J_3-J_2)+(-J_1+J_2)\geq
(J_1-J_2)+(-J_1+J_2)=0.
\end{array}
$$
Due to the arbitrariness of $u(\cdot)$, we conclude that $\bar
u(\cdot)$ is an optimal control process and thus $(\bar{u}(\cdot);
\bar{y}(\cdot), \bar{q}(\cdot), \bar{z}(\cdot))$ is an optimal pair.
\end{proof}
\vspace{1mm}

\section{Applications in  BLQ problems}

In this section, we will apply our stochastic maximum principle to
the so-called BLQ problem, i.e. minimize the following quadratic
cost functional over $u(\cdot) \in \cal A$:
\begin{eqnarray}\label{eq:6.1}
\begin{split}
J(u(\cdot)):=&E\langle My(0),y(0)\rangle
+E\displaystyle\int_0^T\langle E(s)y(s),y(s)\rangle ds
+\sum_{i=1}^dE\displaystyle\int_0^T\langle
F^i(s)q^i(s),q^i(s)\rangle ds\\&+ \sum_{i=1}^\infty
E\displaystyle\int_0^T\langle G^i(s)z^i(s),z^i(s)\rangle ds
+E\displaystyle\int_0^T\langle N(s)u(s),u(s)\rangle ds,
\end{split}
\end{eqnarray}
where the state processes $(y(\cdot), q(\cdot), z(\cdot))$ are the
solution to the controlled linear backward stochastic system as
follows:
\begin{numcases}{}\label{eq:6.2}
dy(t)=-\bigg[A(t)y(t)+\displaystyle\sum_{i=1}^dB^i(t)q^i(t)
+\displaystyle\sum_{i=1}^\infty C^{i}(t)z^i(t)+D(t)u(t)\bigg]dt\nonumber\\
\ \ \ \ \ \ \ \ \ \ \ +\displaystyle\sum_{i=1}^dq^idW^i(t)+\displaystyle\sum_{i=1}^\infty z^idH^{i}(t)\\
y(T)=\xi.\nonumber
\end{numcases}

To study this problem, we need the assumptions on the coefficients
below.
\begin{ass}\label{ass:5.1}
The $\{{\mathscr{F}}_t,0\leq t\leq T\}$-predictable matrix processes
$A:[0,T]\times \Omega \rightarrow \mathbb{R}^{n\times n},
B^i:[0,T]\times \Omega\rightarrow \mathbb{R}^{n\times
n},i=1,2,\cdots,d, C^i:[0,T]\times \Omega \rightarrow
\mathbb{R}^{n\times n},i=1,2,\cdots, D:[0,T]\times \Omega\rightarrow
\mathbb{R}^{n\times m}, E:[0,T]\times \Omega \rightarrow
\mathbb{R}^{n\times n}, F^i:[0,T]\times \Omega\rightarrow
\mathbb{R}^{n\times n}, i=1,2,\cdots d, G^i:[0,T]\times \Omega
\rightarrow \mathbb{R}^{n\times n}, i=1,2,\cdots, N:[0,T]\times
\Omega \rightarrow \mathbb{R}^{m\times m}$ and the
${\mathscr{F}}_T$-measurable random matrix $M:\Omega\rightarrow
\mathbb{R}^{n\times n}$ are uniformly bounded.
\end{ass}
\begin{ass}\label{ass:5.2}
The state weighting  matrix processes $E$, $F^i$, $G^i$, the control
weighting matrix process $N$ and the random matrix $M$ are a.e.
a.s. symmetric and nonnegative. Moreover, $N$ is a.e. a.s. uniformly positive, i.e. $ N\geq
\delta I $ for some positive constant
$\delta$ a.e. a.s. %and almost all $(t, \omega) \in [0, T]\times \Omega$. Also
%assume that the symmetric and nonnegative.
\end{ass}
\begin{ass}\label{ass:5.3}
There is no further constraint imposed on the control processes,
i.e. %the set all admissible control processes is
$$\cal A=\bigg\{u(\cdot)|u(\cdot)\ is\ \mathscr{F}_t-predictable\ with\ values\ in\ \mathbb{R}^{m}\ and\ E\displaystyle\int_0^T|u(t)|^2dt< \infty \}.$$
\end{ass}

From Assumption \ref{ass:5.3}, we know that $\cal A$ is a Hilbert
space. If we denote the norm of $\cal A$ by $\|\cdot\|_{\cal A}$,
then for any control process $u(\cdot)\in \cal A$,
$\|u(\cdot)\|_{\cal A}=E\displaystyle\sqrt{\int_0^T|u(t)|^2dt}$.

Under Assumptions \ref{ass:5.1},  by Lemma \ref{lem:1.3} we first
know that the linear BSDE (\ref{eq:6.2}) in BLQ problem has a unique
solution and thus the BLQ problem is well-defined. Then, under Assumptions \ref{ass:5.1}-\ref{ass:5.3}, we will
demonstrate that BLQ problem has a unique optimal control.
\begin{lem}\label{lem:b4} Under Assumptions
\ref{ass:5.1}-\ref{ass:5.3}, the cost functional $J$ is strictly
convex over $\cal A$ and $\displaystyle\lim_ {\|u(\cdot)\|_{\cal
A}{\rightarrow \infty}}J(u(\cdot))=\infty.$
\end{lem}
\begin{proof} The convexity of the cost functional $J$ over $\cal A$ is obvious. Actually, since
the weighting matrix process $N$ is uniformly positive, $J$ is
strictly convex. In view of  the nonnegative property  of $M, E,
F^i, G^i$ and the strictly positive property of $N$, we have
$$J(u(\cdot))\geq \delta E\displaystyle\int_0^T
|u(t)|^2dt=\delta\|u(\cdot)\|^2_{\cal A}.$$
Therefore, $\displaystyle\lim_ {\|u(\cdot)\|_{\cal A}{\rightarrow
\infty}}J(u(\cdot))=\infty.$
\end{proof}
\vspace{1mm}

\begin{lem}\label{lem:b5} Under Assumptions \ref{ass:5.1}-\ref{ass:5.3}, the cost
functional $J$ is Fr\'{e}chet differentiable over $\cal A$ and its
Fr\'{e}chet derivative $J'$ at any admissible control process
$u(\cdot)\in{\cal A}$ is given by
\begin{eqnarray}\label{tz2}
\begin{split}
\la J'(u(\cdot)), v(\cdot)\ra=&2E\int_0^T\langle E(t)y^u(t),
Y^v(t)\rangle dt+2\sum_{i=1}^dE\int_0^T\langle F^i(t)q^{iu}(t),
Q^{iv}(t)\rangle dt
\\&+2\sum_{i=1}^\infty E\int_0^T\langle G^i(t)z^{iu}(t), Z^{iv}(t)\rangle dt
+2E\int_0^T\langle N(t)u(t), v(t)\rangle dt
\\&+2E\langle My^u(0), Y^v(0)\rangle,   %\forall v(\cdot)\in \cal A
\end{split}
\end{eqnarray}
where $v(\cdot)\in \cal A$ is arbitrary, $(Y^v, Q^v, Z^v)$ is  the
solution of BSDE \eqref{eq:6.2} corresponding to the control process
$v(\cdot)\in \cal A$ and the terminal value $0$, and $(y^u(\cdot),
q^u(\cdot), z^u(\cdot))$ are the state processes corresponding to
the control process $u(\cdot)$.
\end{lem}
\begin{proof} For any $v(\cdot)\in \cal A$, we set
\begin{eqnarray*}
  \begin{split}
\Delta J=&J(u(\cdot)+v(\cdot))-J(u(\cdot))-2E\int_0^T\langle
E(t)y^u(t), Y^v(t)\rangle dt\\&-2\sum_{i=1}^dE\int_0^T\langle
F^i(t)q^{iu}(t), Q^{iv}(t)\rangle dt -2\sum_{i=1}^\infty
E\int_0^T\langle G^i(t)z^{iu}(t), Z^{iv}(t)\rangle dt
\\&-2E\int_0^T\langle N(t)u(t), v(t)\rangle dt-2E\langle My^u(0), Y^v(0)\rangle.
  \end{split}
\end{eqnarray*}
By the definition of cost functional \eqref{eq:6.1}, we have
\begin{eqnarray*}
\begin{split}
\Delta J =&E\langle MY^v(0),Y^v(0)\rangle
+E\displaystyle\int_0^T\langle E(s)Y^v(s),Y^v(s)\rangle ds
+\sum_{i=1}^dE\displaystyle\int_0^T\langle
F^i(s)Q^{iv}(s),Q^{iv}(s)\rangle ds\\&+\sum_{i=1}^\infty
E\displaystyle\int_0^T\langle G^i(s)Z^{iv}(s),Z^{iv}(s)\rangle
ds+E\displaystyle\int_0^T\langle N(s)v(s),v(s)\rangle ds.
\end{split}
\end{eqnarray*}
Then it follows from  Assumption \ref{ass:5.1} and a priori estimate
\eqref{eq:1.5} that
\begin{eqnarray*}
|\Delta J|&\leq&K\bigg[E\sup_{0\leq t\leq
T}|Y^v(t)|^2+E\int_0^T|Q^v(t)|^2dt+E\int_0^T\|Z^v(t)\|^2_{l^2(\mathbb{R}^N)}dt+E\int_0^T|v(t)|^2dt\bigg]\nonumber\\
&\leq&KE\int_0^T|v(t)|^2dt=K\|v(\cdot)\|^2_{\cal A}.
\end{eqnarray*}
%where $K$  is a positive constant.
Consequently, we have
$$\displaystyle\lim_ {\|v(\cdot)\|_{\cal A}
{\rightarrow 0}}\frac{|\Delta J|}{\|v(\cdot)\|_{\cal A}}=0,$$
%Thus
%from the definition of Fr\'{e}chet derivative,
which implies that $J$ is Fr\'{e}chet differentiable and its
Fr\'{e}chet derivative $J'$ is given by \eqref{tz2}.
\end{proof}
\vspace{1mm}

The strict convexity and the Fr\'{e}chet differentiability of $J$
deduced from Lemmas \ref{lem:b4}-\ref{lem:b5} lead to the lower
semi-continuity of $J$, thus the following lemma is applicable to
$J$ and $\cal A$ in our BLQ problem.
\begin{lem}\label{tz4} (Proposition 1.2 of Chapter II in \cite{EkTe76})
Let $\cal A$ be a reflexive Banach space and $J:\cal
A\mapsto\mathbb{R}^1$ be a convex function. Assume that $J$ is lower
semi-continuous and proper, and consider the minimization problem
$$\inf_{u\in \cal A} J(u).$$
If the function $J$ is coercive over $\cal A$, i.e.
$$\lim_{\|u\|_{\cal A}\to\infty}J(u)=\infty,$$
then the minimization problem has at least one solution. Moreover,
if $J$ is strictly convex over $\cal A$, then the minimization
problem has a unique solution.
\end{lem}
By Lemma \ref{tz4} we can immediately conclude %that the stochastic
%LQ problem has a unique optimal control.
\begin{thm}\label{them:b1}
Under Assumptions \ref{ass:5.1}-\ref{ass:5.3}, BLQ problem has a
unique optimal control. \end{thm}

In what follows, we will utilize the stochastic maximum principle to
study the dual representation of the optimal control to BLQ problem
and construct its stochastic Hamilton system. As in  section 4, we
first introduce the adjoint forward equation corresponding to an
admissible pair $(u(\cdot);
 y(\cdot),q(\cdot),z(\cdot))$:
\begin{numcases}{}\label{eq:7.6}
dk(t)=\bigg(A^*(t)k(t)-2E(t)y(t)\bigg)dt
+\displaystyle\sum_{i=1}^d\bigg(B^{i*}(t)k^i(t)-2F^i(t)q^i(t)\bigg)dW^i(t)\nonumber\\
\ \ \ \ \ \ \ \ \ \ \ \ +\displaystyle\sum_{i=1}^\infty
\bigg(C^{i*}(t)k(t)-2G^{i}(t)z^i(t)\bigg)dH^{i}(t)\\
k(0)={-2My(0)}.\nonumber
\end{numcases}
Also we define the Hamiltonian function $H:[0,T]\times
\Omega\times\mathbb{R}^n\times\mathbb{R}^{n\times d}\times
l^2(\mathbb{R}^n)\times U\times\mathbb{R}^n\rightarrow \mathbb{R}^1$
by
\begin{eqnarray}\label{eq:6.7}
\displaystyle H(t,y,q,z,u,k)&=&-\bigg\langle k,
A(t)y+\displaystyle\sum_{i=1}^dB^i(t)q^i+\displaystyle\sum_{i=1}^\infty
C^{i}(t)z^i+D(t)u\bigg\rangle\\
&&+\langle E(t)y, y\rangle +\displaystyle\sum_{i=1}^d\langle
F^i(t)q^i,q^i\rangle +\displaystyle\sum_{i=1}^\infty \langle
G^i(t)z^i,z^i\rangle +\langle N(t)u, u\rangle.\nonumber
\end{eqnarray}
Then the adjoint equation can be rewritten as a Hamiltonian form:
\begin{numcases}{}\label{5.8}
dk(t)=-H_y(t, {y}(t), {q}(t), {z}(t),
{u}(t),k(t))dt-\displaystyle\sum_{i=1}^d H_q^{i} (t, {y}(t), {q}(t),
{z}(t), {u}(t),k(t))dB^i(t)\nonumber\\
\ \ \ \ \ \ \ \ \ \ \ -\displaystyle\sum_{i=1}^\infty
H_{z^i}(t, {y}(t-), {q}(t), {z}(t), {u}(t),k(t))dH^i(t)\\
k(0)=-2My(0).\nonumber
\end{numcases}
Under Assumption \ref{ass:5.1}, for each admissible pair
$({u}(\cdot); {y}(\cdot), {q}(\cdot), {z}(\cdot))$, by Lemma
\ref{lem:1.1} the adjoint equation (\ref{5.8}) has a unique solution
$k(\cdot)$.

It is time to give the  the dual characterization of the optimal
control.
\begin{thm}\label{thm:b2}
Under Assumptions \ref{ass:5.1}-\ref{ass:5.3}, BLQ problem has a
unique optimal control and the optimal control is given by
\begin{eqnarray} \label{eq:6.9}
  \begin{split}
    u(t) =-\frac{1}{2}N^{-1}(t)D^*(t)k(t-),\ \ a.e.\ a.s.,
     \end{split}
\end{eqnarray}
where $k(\cdot)$ is the unique solution of the adjoint equation
\eqref{eq:7.6} (or equivalently, \eqref{5.8}) corresponding to the
optimal pair $(u(\cdot); y(\cdot), q(\cdot),z(\cdot))$.
\end{thm}
\begin{proof} By Theorem \ref{them:b1}, we know the existence and
uniqueness of optimal control to BLQ problem and denote the optimal
control by $u(\cdot)$. We only need to prove $u$ has an expression
as in (\ref{eq:6.9}). For this, let $( y(\cdot), q(\cdot),z(\cdot))$
be the optimal state processes corresponding to $u(\cdot)$ and
$k(\cdot)$ be the unique solution of the adjoint equation
\eqref{5.8} corresponding to the optimal pair $(u(\cdot); y(\cdot),
q(\cdot),z(\cdot))$. By the necessary optimality condition
\eqref{eq:44} and Assumption \ref{ass:5.3}, we have
$$H_u(t,y(t-),q(t),z(t), u(t), k(t-))=0,\ \ a.e.\ a.s.$$
Noticing the definition of $H$ in \eqref{eq:6.7}, we get
$$2N(t)u(t)+D^{*}(t)k(t-)=0,\ \ a.e.\ a.s.$$
Then the claim that the unique optimal control $u(\cdot)$ satisfies
\eqref{eq:6.9} follows.
\end{proof}
\vspace{1mm}

Finally we introduce the so-called stochastic Hamilton system which
consists of the state equation \eqref{eq:6.2}, the adjoint equation
\eqref{eq:7.6} (or equivalently, \eqref{5.8}) and the dual
representation \eqref{eq:6.9}:
\begin{numcases}{}\label{eq:f333}
dy(t)=-\bigg(A(t)y(t)+\sum_{i=1}^dB^i(t)q^i(t) +\sum_{i=1}^\infty
C^{i}(t)z^i(t)+D(t)u(t)\bigg)dt\nonumber\\
\ \ \ \ \ \ \ \ \ \ \
+\displaystyle\sum_{i=1}q^idW^i(t)+\displaystyle\sum_{i=1}^\infty
z^idH^{i}(t)\nonumber\\
y(T)=\xi,\nonumber\\
dk(t)=\bigg(A^*(t)k(t)-2E(t)y(t)\bigg)dt
+\displaystyle\sum_{i=1}^d\bigg(B^{i*}(t)k^i(t)-2F^i(t)q^i(t)\bigg)dW^i(t)\nonumber\\
\ \ \ \ \ \ \ \ \ \ \ \ +\displaystyle\sum_{i=1}^\infty
 \bigg(C^{i*}(t)k(t)-2G^{i}(t)z^i(t)\bigg)dH^{i}(t)\\
k(0)=-2M y(0),\nonumber\\
u_t=-\frac{1}{2}N^{-1}(t)D^*(t)k(t-).\nonumber
\end{numcases}
Clearly this is a fully coupled  forward-backward stochastic
differential equation (FBSDE) driven by $d$-dimensional Brownian
motion $W$ and Teugel's martingales $\{H^i\}_{i=1}^\infty$, and its
solution is a stochastic processes quaternary $(k(\cdot), y(\cdot),
q(\cdot), z(\cdot))$.

\begin{thm} Under Assumptions \ref{ass:5.1}-\ref{ass:5.3}, the
stochastic Hamilton system \eqref{eq:f333} has a unique solution
$(k(\cdot), y(\cdot), q(\cdot), z(\cdot))\in S_{\mathscr{F}}^2(0,
T;\mathbb{R}^n)\times S_{\mathscr{F}}^2(0, T;\mathbb{R}^n)\times
M_{\mathscr{F}}^2(0,T; \mathbb{R}^{n\times d}) \times
l_{\mathscr{F}}^2(0,T;\mathbb{R}^n)$, where $u(\cdot)$ is the
optimal control of BLQ problem and $(y(\cdot), q(\cdot),
z(\cdot))$ are its corresponding optimal state. Moreover, %the
%following a priori
%estimate holds
\begin{eqnarray}
\label{eq:f334}
\begin{split}
&\displaystyle E\sup_{0\leqslant t\leqslant
T}|k(t)|^2+E\displaystyle\sup_{0\leq t\leq T}|y(t)|^2
+E\int_0^T|q(t)|^2dt+E\int_0^T||z(t)||^2_{l^2({\mathbb{R}^n})}dt\leqslant
K E{{|\xi|^2}}.
\end{split}
\end{eqnarray}
\end{thm}
\begin{proof} The existence result follows from Theorem
\ref{thm:b2} and the uniqueness result is obvious once a priori
estimate \eqref{eq:f334} holds. But noticing Assumptions
\ref{ass:5.1}-\ref{ass:5.3} and using Lemmas \ref{lem:1.2} and
\ref{lem:1.4}, we can deduce \eqref{eq:f334} immediately.
\end{proof}
\vspace{1mm}

In summary, the stochastic Hamilton system \eqref{eq:f333}
completely characterize the optimal control of BLQ problem in this
section. Therefore, solving BLQ problem is equivalent to solving the
stochastic Hamilton system, moreover, the unique optimal control of
the stochastic Hamilton system can be given explicitly by
\eqref{eq:6.9}.

%\bibliography{wenxianku}

\begin{thebibliography}{10}

\bibitem{BEE03}
K.~Bahlali, M.~Eddahbi and E.~Essaky.
\newblock BSDE associated with L\'{e}vy processes and application to PDIE.
\newblock {\em J. Appl. Math. Stochastic Anal.}, 16:1--17, 2003.

\bibitem{BGM10}
K.~Bahlali, B.~Gherbal and B.~Mezerdi.
\newblock Existence and optimality conditions in stochastic control of linear
BSDEs.
\newblock {\em Random Oper. Stoch. Equ.}, 18:185--197, 2010.

\bibitem{Bah2010}
S.~Bahlali.
\newblock Stochastic controls of backward systems.
\newblock {\em Random Oper. Stoch. Equ.}, 18:125--140, 2010.

\bibitem{Bism73}
J.-M. Bismut.
\newblock Conjugate convex functions in optimal stochastic control.
\newblock {\em J. Math. Anal. Appl.}, 44:384--404, 1973.

\bibitem{DoZh99}
N.~Dokuchaev and X.~Y. Zhou.
\newblock Stochastic controls with terminal contingent conditions.
\newblock {\em J. Math. Anal. Appl.},
238:143--165, 1999.

\bibitem{EkTe76}
I.~Ekeland and R.~T\'{e}mam.
\newblock {\em Convex Analysis and Variational Problems}.
\newblock Amsterdam: North-Holland, 1976.

\bibitem{KPQ97}
N.~El~Karoui, S.~Peng and M.~C. Quenez.
\newblock Backward stochastic differential equations in finance.
\newblock {\em Math. Finance}, 7:1--71, 1997.

\bibitem{Otm06}
M.~El~Otmani.
\newblock Generalized BSDE driven by a L\'{e}vy process.
\newblock {\em J. Appl. Math. Stoch. Anal.}, 25pp, 2006.

\bibitem{Otm08}
M.~El~Otmani.
\newblock Backwark stochastic differnetial equations associated with L\'{e}vy
processes and partial integro-differential eqiations.
\newblock {\em Commun. Stoch. Anal.}, 2:277--288, 2008.

\bibitem{AnZh01}
A.~E.~B.~Lim and X.~Y. Zhou.
\newblock Linear-quadratic control of backward stochastic differential
equations.
\newblock {\em SIAM J. Control Optim.}, 40:450--474, 2001.

\bibitem{AnZh02}
A.~E.~B.~Lim and X.~Y. Zhou.
\newblock Optimal control of linear backward stochastic differential equations
with a quadratic cost criterion.
\newblock {\em Stochastic Theory and Control, Lecture Notes in Control and
Inform. Sci.}, 280:301--317, \newblock Berlin: Springer, 2002.

\bibitem{MeTa08}
Q.~Meng and M.~Tang.
\newblock Necessary and sufficient conditions for optimal control of stochastic systems
associated with L\'{e}vy processes.
\newblock {\em Sci. China Ser. F}, 52:1982--1992, 2009.

\bibitem{MiTa08}
K.~Mitsui and Y.~Tabata.
\newblock A stochastic linear-quadratic problem with L\'{e}vy processes and its
application to finance.
\newblock {\em Stochastic Proecss. Appl.}, 118:120--152, 2008.

\bibitem{NuSc}
D.~Nualart and W.~Schoutens.
\newblock Chaotic and predicatable representation for L\'{e}vy processes.
\newblock {\em Stochastic Process. Appl.}, 90:109--122, 2000.

\bibitem{NuSc01}
D.~Nualart and W.~Schoutens.
\newblock Backward stochastic differential equations and Feynman-Kac formula for L\'{e}vy processes with
applications in finance.
\newblock {\em Bernouli}, 7:761--776, 2001.

\bibitem{PaPe90}
E.~Pardoux and S.~Peng.
\newblock Adapted solution of a backward stochastic differential equation.
\newblock {\em Systems Control Lett.}, 14:55--61, 1990.

\bibitem{ReFa09}
Y.~Ren and X.~L. Fan.
\newblock Refelected backward stochastic differential equations driven by a
L\'{e}vy process.
\newblock {\em ANZIAM J.}, 50:486--500, 2009.

\bibitem{ReOt10}
Y.~Ren and M.~El~Otmani.
\newblock Generalized reflected BSDEs driven by a L\'{e}vy process and an obstacle
problem for PDIEs with a nonlinear Neumann boundary condition.
\newblock {\em J. Comput. Appl. Math.},
233:2027--2043, 2010.

\bibitem{TaWu09}
H.~Tang and Z.~Wu.
\newblock Stochastic differential equations and stochastic linear quadratic
optiaml control problem with L\'{e}vy processes.
\newblock {\em J. Syst. Sci. Complex.}, 22:122--136, 2009.

\bibitem{Tali94}
S.~Tang and X.~Li.
\newblock Necessary conditions for optimal control of stochastic systems with
random jumps.
\newblock {\em SIAM J. Control Optim.}, 32:1447--1475, 1994.

\end{thebibliography}
%\bibliographystyle{plain}

\end{document}